\documentclass[11pt]{article}
\usepackage{amsfonts}
\usepackage{amsmath}
\usepackage{amsthm}
\usepackage{amscd}
\usepackage{amsmath}
\usepackage{amssymb}
\usepackage{amscd}
\usepackage{amsfonts}
\usepackage{amsbsy}
\usepackage{graphicx}

\usepackage{booktabs}
\usepackage{makecell}
\usepackage{enumitem}
\usepackage{tikz}
\usepackage{booktabs}
\usepackage{multirow}
\usepackage{diagbox}
\usepackage{appendix}
\usepackage{mathrsfs}

\usepackage{caption}
\usepackage{float}
\usepackage{subcaption}
\usepackage{authblk}
\usepackage{caption}
\usepackage{subcaption}

\usepackage{hyperref}
\hypersetup{colorlinks=true,citecolor=blue,%
        linkcolor=blue,urlcolor=blue,bookmarksnumbered=true,
        bookmarksopen=true,bookmarksopenlevel=1,breaklinks=true}
\def\min{\mathsf{min}}

\newtheorem{theorem}{Theorem}[section]
\newtheorem{proposition}[theorem]{Proposition}
\newtheorem{lemma}[theorem]{Lemma}
\newtheorem{corollary}[theorem]{Corollary}

\theoremstyle{definition}

\def\min{\mathsf{min}}
\numberwithin{equation}{section}

\numberwithin{figure}{section}
\usepackage{booktabs}
\usepackage{subcaption}

\begin{document}
\title
{Linear multiplicative noise destroys a two-dimensional attractive compact manifold of three-dimensional Kolmogorov systems}

\author{Dongmei Xiao\footnote{School of Mathematical Sciences, CMA-Shanghai, Shanghai Jiao Tong University, Shanghai 200240, China, e-mail: \href{mailto:xiaodm@sjtu.edu.cn}{xiaodm@sjtu.edu.cn}}\quad Shengnan Yin\footnote{Huantai No.1 Middle School, 256400, Shandong, China, e-mail: \href{mailto:snyin06@163.com}{snyin06@163.com}}\quad Chenwan Zhou\footnote{School of Mathematical Sciences, CMA-Shanghai, Shanghai Jiao Tong University, 200240, Shanghai, China, e-mail: \href{mailto:daydayupzcw@sjtu.edu.cn}{daydayupzcw@sjtu.edu.cn}}}

 \date{}

\maketitle
\begin{abstract}
In the paper we first characterize three-dimensional Kolmogorov systems possessing a two-dimensional invariant sphere in $\mathbb{R}^3$, then establish a global attracting
criterion for this invariant sphere in $\mathbb{R}^3$ except the origin, and give global dynamics with isolated equilibria on the sphere. Finally, we consider the persistence of the attractive invariant sphere under the perturbation induced by linear multiplicative Wiener noise. It is shown that suitable noise intensity can destroy the sphere and lead to bifurcation of stationary measures.

\medskip
\noindent
{\bf MSC2020 subject classifications:} 34C05, 34C45, 34F10, 37H20.

\end{abstract}
\section{Introduction}
Kolmogorov systems was first proposed by Kolomogorov in \cite{kl} to describe  the growth rate of populations in a community of $n$ interacting species in population dynamics, which is defined by the following  ordinary differential equations,
\begin{equation}\label{sys1.1}
\frac{dx_i(t)}{dt}=x_i(t)G_i(x_1(t),...,x_n(t)), i=1,...,n,
\end{equation}
where $(x_1, ...,x_n) \in \mathbb{R}_+^n=\{x\in \mathbb{R}^n: x=(x_1,...,x_n), x_i\geq 0\}$,  and $G_i(x_1, ..., x_n)$ is continuous differentiable, $i=1,...,n$. The dynamical behavior of system \eqref{sys1.1} indicates the changing law of populations in the community, and the extinction and coexistence of species correspond to the existence of some attractive invariant sets for system \eqref{sys1.1}. Hence, the study on the existence and structure of attractive invariant set of system \eqref{sys1.1} is a central topic in population dynamics.  Arneodo et. al. in \cite{StrangeAttractor} observed chaotic behaviour of a class of three-dimensional Lotka-Volterra systems by numerical simulation. Note that Lotka-Volterra system is Kolomogorov system \eqref{sys1.1} with linear polynomials $G_i(x_1(t),...,x_n(t))$ for $i=1,2,\cdots,n$.
Almost at the same time  Busse and his collaborators in \cite{busse1, busse2, busse3}  studied  the turbulent convection in a fluid layer by a three dimensional Kolmogorov system and pointed out that the occurrence of turbulent  was approximately described by a manifold in the mode
space and disturbances may come from the external environment noise.  It is well known that the existence of invariant manifolds plays important role in understanding global dynamics of dynamical systems. Li in \cite{Li-Michael1, Li-Michael2} found that the existence of invariant manifolds  has some implications to exclude the existence of periodic solutions, which greatly improved higher dimensional Dulac criterion. On the other hand, the external environment noise induce a random perturbation of dynamical systems. The change of steady measures and persistence of invariant manifolds under random perturbations attract many mathematicians, see \cite{Arnold, jiang1, HJLY1, HJLY2, HJLY4, Zhao} and reference therein.
Inspired by aforementioned remarkable works, we consider two problems in mathematics:  what kinds of three-dimensional Kolmogorov systems have two-dimensional attractive invariant compact  manifold in $\mathbb{R}^3$? What happens the two-dimensional attractive invariant compact  manifold under noise perturbation?

The aim of this paper is to explore the conditions of three-dimensional polynomial Kolmogorov systems having two-dimensional attractive invariant compact manifold  $\mathbb{S}^2$ (the euclidean unit sphere) in $\mathbb{R}^3$, study global dynamics of this Kolmogorov system and it's stochastic dynamics under the perturbation of linear multiplicative Wiener noise, and discuss bifurcation of stationary measures  when the noise intensity changes. 

In Section \ref{3d determinisitc}, using Darboux theory, we give the sufficient and necessary conditions for  three-dimensional cubic polynomial Kolmogorov systems possessing invariant compact manifold $\mathbb{S}^2$ (see Proposition \ref{thm1}), and establish a global attracting criterion for this invariant sphere in $\mathbb{R}^3\setminus\{O\}$ by Lyapunov function (see  Theorem \ref{attractor}). Different from the results in \cite{Li-Michael1}, we find that the  Kolmogorov system has either periodic orbits or non-periodic orbits on this invariant manifold  $\mathbb{S}^2$ (see Theorem \ref{global dynamics} and Figure \ref{global dynamcis}). 

Further, in Section \ref{stochastic 3d}, we consider the  Kolmogorov system with attractor $\mathbb{S}^2$ under the perturbation
induced by linear multiplicative Wiener noise. Combined Lyapunov function coming from the structure of the associated deterministic system with the Doss-Sussmann transform in \cite{Doss, Sussmann}, we prove that there exists a threshold $\sigma_0$ such that when the noise intensity $\sigma>\sigma_0$, the noise destroys the attracting invariant sphere $\mathbb{S}^2$. Moreover, the change of the noise intensity $\sigma$ in neighborhoods of some thresholds leads to transitions of stationary measures, that is,  there exists another threshold $\sigma_1<\sigma_0$ such that when $\sigma>\sigma_0$, there is a unique stationary measure; while $0<\sigma_1<\sigma<\sigma_0$ leads to at least two stationary measures; and the weaker noise $\sigma<\sigma_1$ causes at least four stationary measures (see Theorem \ref{stochastic dynamics}).

It is worth noting that there have been many associated works on additive noise such as Crauel and Flandoli \cite{Crauel-Flandoli}, Brzez\'niak et. al. \cite{braza} and references therein. Compared with additive noise, there is less study on multiplicative noise. Recently, we studied stochastic bifurcations of a three-dimensional Kolmogorov system with the same intrinsic growth rate under the perturbation of linear multiplicative noise, see \cite{zhou2024stochastic}. Unfortunately, the methods used in \cite{Crauel-Flandoli, braza, zhou2024stochastic}  can not be directly  applied to deal with stochastic bifurcations of our three-dimensional Kolmogorov system with the different intrinsic growth rate by linear multiplicative noise perturbing. Doss-Sussmann transform
and  Lyapunov function are our key tools in this paper.

\section{Three-dimensional polynomial Kolmogorov systems with an invariant sphere} \label{3d determinisitc}
In this section, we consider three-dimensional polynomial Kolmogorov differential systems
\begin{equation}\label{Kol3}
\frac{dx_i}{dt}=x_iG_i(x_1,x_2,x_3), \ i=1,2,3,
\end{equation}
where $(x_1,x_2,x_3)\in \mathbb{R}^3$, $x_iG_i(x_1,x_2,x_3), i=1,2,3,$ are coprime polynomials, and the degree of system \eqref{Kol3} is denoted by $m=\max_{i=1,2,3}{\textrm{deg}}\left(x_iG_i(x_1,x_2,x_3)\right)$.
We first give the necessary condition for system \eqref{Kol3} having an isolated invariant sphere $\mathbb{S}^2$ as follows.

\begin{proposition}\label{m2}
If three-dimensional system \eqref{sys1.1} has an isolated invariant sphere $\mathbb{S}^2$, then the degree $m$ of system \eqref{Kol3} satisfies $m\ge 3$.
\end{proposition}
\begin{proof}
Since $\mathbb{S}^2$ is an isolated invariant sphere, we have that  $F(x_1,x_2,x_3)=x_1^2+x_2^2+x_3^2-1$ is a Darboux polynomial of system \eqref{Kol3} with a nonzero cofactor $K(x_1,x_2,x_3)$ by Darboux Thoerem, where
we say that $F(x_1,x_2,x_3)$ is a {\it Darboux polynomial} of system \eqref{Kol3} if there exists a
polynomial $K(x_1,x_2,x_3)$ of degree at most $m-1$, called the {\it cofactor} of $F(x_1,x_2,x_3)$, such that
\begin{equation}\label{eq1}
\frac{dF(x_1,x_2,x_3)}{dt}\left|_{\eqref{Kol3}}\right.=\sum_{i=1}^3\frac{\partial F}{\partial x_i}x_iG_i(x_1,x_2,x_3)=F(x_1,x_2,x_3)K(x_1,x_2,x_3),
\end{equation}
see \cite{Darboux}. Obvious, $F(x_1,x_2,x_3)$ is a first integral of system \eqref{Kol3} if the cofactor $K(x_1,x_2,x_3)\equiv 0$.

Assume that $K(x_1,x_2,x_3)=\sum_{j=0}^{m-1} K_j(x_1,x_2,x_3)$ in which $K_j(x_1,x_2,x_3)$ is a homogeneous polynomial of $(x_1,x_2,x_3)$ with degree $j$.
 Then \eqref{eq1} becomes
\begin{equation}\label{Least_degree}
\sum_{i=1}^32x_i^2G_i(x_1,x_2,x_3)=-\sum_{j=0}^{m-1} K_j(x_1,x_2,x_3)+(x_1^2+x_2^2+x_3^2)\left(\sum_{j=0}^{m-1} K_j(x_1,x_2,x_3)\right).
\end{equation}
It can be seen that the polynomial of left part of \eqref{Least_degree} does not have constant term and linear term. By comparing the coefficients of the polynomials in the same power at \eqref{Least_degree}, one has $K_0=0$ and $K_1(x_1,x_2,x_3)\equiv 0$. This means the degree of cofactor $K(x_1,x_2,x_3)$ is at least two. As a result, $m\geq 3$.
\end{proof}

Proposition \ref{m2} tells us that  three-dimensional Lotka-Volterra systems can not have an isolated invariant sphere $\mathbb{S}^2$. In order to avoid the tedious calculation, we consider the conditions for the following cubic polynomial Kolmogorov differential systems possessing isolated invariant sphere $\mathbb{S}^2$.
\begin{equation}\label{sys1.2}
\begin{cases}
\frac{dx_1}{dt}=x_1\,\big(r_1+\sum_{i=1}^3a_ix_i+\sum_{1\le i\le j\le 3}a_{ij}x_ix_j\big),\\
\frac{dx_2}{dt}=x_2\,\big(r_2+\sum_{i=1}^3b_ix_i+\sum_{1\le i\le j\le 3}b_{ij}x_ix_j\big),\\
\frac{dx_3}{dt}=x_3\,\big(r_3+\sum_{i=1}^3c_ix_i+\sum_{1\le i\le j\le 3}c_{ij}x_ix_j\big),\\
\end{cases}
\end{equation}
 where $r_i,a_i,b_i$, $c_i$,  $a_{ij}$, $b_{ij}$ and $c_{ij}$ are real parameters, here $i,j\in \{1,2,3\}$.

\subsection{Cubic polynomial Kolmogorov differential systems with an attractive invariant sphere $\mathbb{S}^2$}

Now we characterize system \eqref{sys1.2} having  an invariant sphere $\mathbb{S}^2$ in $\mathbb{R}^3$ as follows.

\begin{proposition}\label{thm1}
System \eqref{sys1.2} has an invariant sphere $\mathbb{S}^2$ in $\mathbb{R}^3$ if and only if
\begin{equation}\label{SNC}
\begin{cases}
a_i=b_i=c_i=0, \ i=1,2,3,\\
a_{ij}=b_{ij}=c_{ij}=0, \ i\not=j,\\
a_{11}=-r_1, a_{22}=-(r_1+r_2+b_{11}),\\
b_{22}=-r_2, b_{33}=-(r_2+r_3+c_{22}),\\
c_{11}=-(r_1+r_3+a_{33}),  c_{33}=-r_3.\\
\end{cases}
\end{equation}
\end{proposition}
\begin{proof}
Assume that $\mathbb{S}^2$ is an invariant sphere. Then one has that \eqref{eq1} holds with cofactor  $K(x_1,x_2,x_3)$ of degree $2$. Moreover, it follows from \eqref{Least_degree} that $K(x_1,x_2,x_3)=K_2(x_1,x_2,x_3)$, where $K_2(x_1,x_2,x_3)$ is a homogeneous polynomial with degree 2. Thus, equation \eqref{eq1} can be written as
\begin{equation}\label{eq2}
\begin{aligned}
K_2(x_1,x_2,x_3)(x_1^2+x_2^2+x_3^2-1)&=2(r_1x_1^2+r_2x_2^2+r_3x_3^2)\\
&+2x_1^2(\sum_i^3a_ix_i)+2x_2^2(\sum_i^3b_ix_i)+2x_3^2(\sum_i^3c_ix_i)\\
&+2x_1^2(\sum_{1\leq i\leq j\leq3}a_{ij}x_ix_j)+2x_2^2(\sum_{1\leq i\leq j\leq3}b_{ij}x_ix_j)\\
&+2x_3^2(\sum_{1\leq i\leq j\leq3}c_{ij}x_ix_j).
\end{aligned}
\end{equation}
By comparing the coefficients of the polynomials in the same power of  equality \eqref{eq2},
we  immediately have
\begin{equation}\label{SNC2}
\begin{aligned}
K_2(x_1,x_2,x_3)=-2r_1x_1^2-2r_2x_2^2-2r_3x_3^2, \\
a_i=b_i=c_i=0, \ \ i=1,2,3, \\
a_{ij}=b_{ij}=c_{ij}=0, \ \ i\neq j.
\end{aligned}
\end{equation}
This yields that
\begin{equation}\label{SNC1}
\begin{aligned}
-2(r_1x_1^2+r_2x_2^2+r_3x_3^2)(x_1^2+x_2^2+x_3^2)&=
2x_1^2(a_{11}x_1^2+a_{22}x_2^2+a_{33}x_3^2)\\
&+2x_2^2(b_{11}x_1^2+b_{22}x_2^2+b_{33}x_3^2)\\
&+2x_3^2(c_{11}x_1^2+c_{22}x_2^2+c_{33}x_3^2).
\end{aligned}
\end{equation}
Hence, conditions \eqref{SNC} are true by comparing the coefficients of the polynomials in the same power of the equality \eqref{SNC1}.

 On the contrary, if system \eqref{sys1.2} satisfies condition \eqref{SNC}, then one can check that
 $$
 \frac{dF(x_1,x_2,x_3)}{dt}\left|_{\eqref{Kol3}}\right.=F(x_1,x_2,x_3)K(x_1,x_2,x_3),
 $$
 where $K(x_1,x_2,x_3)=-2r_1x_1^2-2r_2x_2^2-2r_3x_3^2$. So the set  $\mathbb{S}^2=\{(x_1,x_2,x_3)\in \mathbb{R}^3:\ F(x_1,x_2,x_3)=0\}$ is invariant by the flow of system \eqref{sys1.2},
 which implies that $\mathbb{S}^2$ is an invariant sphere. This completes the proof.
\end{proof}
Note that $K(x_1,x_2,x_3)\equiv 0$ if and only if $r_1^2+r_2^2+r_3^2=0$, which leads that $F(x_1,x_2,x_3)$ is a first integral of system \eqref{sys1.2}. Therefore, we have
\begin{corollary}
System \eqref{sys1.2} has an isolated invariant sphere $\mathbb{S}^2$ in $\mathbb{R}^3$ if and only if both \eqref{SNC}  and $r_1^2+r_2^2+r_3^2\neq0$ hold.
\end{corollary}

For simplicity of notations, let
$$
\alpha=\alpha_1, \beta=\alpha_2, \gamma=\alpha_3, d_1=b_{11}, d_2=a_{33}, d_3=c_{22}.
$$
Then system \eqref{sys1.2} with an invariant sphere $\mathbb{S}^2$ can be written as
\begin{equation}\label{sys1}
\begin{cases}
\frac{dx_1}{dt}=x_1\,\big(\alpha_1-\alpha_1\,x_1^2-(\alpha_1+\alpha_2+d_1)\,x_2^2+d_2\,x_3^2\big),\\
\frac{dx_2}{dt}=x_2\,\big(\alpha_2+d_1\,x_1^2-\alpha_2\,x_2^2-(\alpha_2+\alpha_3+d_3)\,x_3^2\big),\\
\frac{dx_3}{dt}=x_3\,\big(\alpha_3-(\alpha_3+\alpha_1+d_2)\,x_1^2+d_3\,x_2^2-\alpha_3\,x_3^2\big),\\
\end{cases}
\end{equation}
where $\alpha_i$ and $d_i$, $i=1,2,3$ are real parameters.

Next theorem gives the necessary and sufficient conditions of system \eqref{sys1} has a global attractor $\mathbb{S}^2$ in $\mathbb{R}^3\setminus\{O\}$.
\begin{theorem}\label{attractor}The invariant sphere $\mathbb{S}^2$ is a global attractor of system \eqref{sys1} in $\mathbb{R}^3\setminus\{O\}$ if and only if $\alpha_i>0, i=1,2,3$.
\end{theorem}
\begin{proof}
    Note that the origin $O$ is an equilibrium of system \eqref{sys1}
and all three eigenvalues of the Jacobian matrix at $O$ are $\alpha_1, \alpha_2, \alpha_3$. And so,
$O$ is a local repeller (attractor) of system \eqref{sys1} if $\alpha_i>0$ ($\alpha_i<0$, resp.) for all $i=1,2,3$. And $O$ is a degenerate equilibrium if at least one of $\alpha_i, i=1,2,3$ is zero. By straightforward computations, if $\alpha_i=0$, then the positive $x_i$-axis is filled with equilibria.
 This leads that there exist some $x_0\in \mathbb{R}^3 \setminus \{O\}$
such that $\omega_d (x_0)\notin  \mathbb{S}^2$ if $\alpha_i \le 0$, $i=1,2,3$. Therefore, $\alpha_i >0$, $i=1,2,3$ if $\mathbb{S}^2$ is a global attractor in $\mathbb{R}^3\setminus \{O\}$.

On the other hand, if $\alpha_i>0$, $i=1,2,3$, then $O$ is a local repeller of system \eqref{sys1}.
Hence,
for any $x_0\in \mathbb{R}^3\setminus\{O\}$
there exists a constant $c(x_0)>0$ such that the solution $\Psi(t, x_0)$ of system \eqref{sys1} passing through $x(0)=x_0$ satisfies
\begin{equation}\label{estimate0}
   \inf\limits_{t\geq 0} \|\Psi(t, x_0)\|\geq c(x_0)
   >0.
\end{equation}

Note that $\mathbb{S}^2$ is an invariant sphere of system \eqref{sys1}. Let us define
$$L(x):=x_1^2+x_2^2+x_3^2-1,\,\, x=(x_1, x_2, x_3)\in \mathbb{R}^3,$$
and without loss of generality, we assume that $\alpha_1\ge \alpha_2\ge \alpha_3>0$.
Then, by some computations,
$\forall \ x_0\in \mathbb{R}^3$,
\begin{align}\label{Leq}
\frac{dL(\Psi(t,x_0))}{dt}|_{\eqref{sys1}}&=-2(\alpha_1x_1^2+\alpha_2x_2^2+\alpha_3x_3^2)L(\Psi(t, x_0)) \\ \notag
&\le -2\alpha_3\|\Psi(t, x_0)\|^2L(\Psi(t, x_0))
\end{align}
Taking into account \eqref{Leq} and \eqref{estimate0} we have
$$\|L(\Psi(t, x_0))\|\leq \|L(x_0)\|\exp\{\int_0^t -2\alpha_3 {c^2(x_0)}ds\},\ \forall t\ge 0, x_0\in \mathbb{R}^3\setminus \{O\}.$$
Thus,
$$\displaystyle \lim_{t\rightarrow +\infty}\|L(\Psi(t, x_0))\|=0.$$
This yields that for any $x_0\in \mathbb{R}^3\setminus \{O\}$, $\omega_d(x_0)\subseteq \mathbb{S}^2$. Therefore, the invariant sphere $\mathbb{S}^2$
is a global attractor of system \eqref{sys1} in $\mathbb{R}^3\setminus \{O\}$.
\end{proof}

Therefore,  from Proposition \ref{thm1} and  Theorem \ref{attractor},  we know that the three-dimensional cubic polynomial Kolmogorov system \eqref{sys1.2}  has a global attractor in $\mathbb{R}^3\setminus\{O\}$, which is exactly $\mathbb{S}^2$, if and only if it can be written as 
\begin{equation}\label{sys0}
\begin{cases}
\frac{dx_1}{dt}=x_1\,\big(\alpha_1-\alpha_1\,x_1^2-(\alpha_1+\alpha_2+d_1)\,x_2^2+d_2\,x_3^2\big),\\
\frac{dx_2}{dt}=x_2\,\big(\alpha_2+d_1\,x_1^2-\alpha_2\,x_2^2-(\alpha_2+\alpha_3+d_3)\,x_3^2\big),\\
\frac{dx_3}{dt}=x_3\,\big(\alpha_3-(\alpha_3+\alpha_1+d_2)\,x_1^2+d_3\,x_2^2-\alpha_3\,x_3^2\big),\\
\end{cases}
\end{equation}
where $\alpha_i>0$, $i=1,2,3.$

\subsection{Global dynamics of system \eqref{sys0} with isolated equilibria}

Global dynamics of system \eqref{sys0} has been studied in \cite{zhou2024stochastic} when $0<\alpha_1=\alpha_2=\alpha_3$. In this subsection we investigate the topological classification of global dynamics of system \eqref{sys0} when at least two of $\alpha_1$, $\alpha_2$ and $\alpha_3$ are not equal and all of equilibria of system \eqref{sys0} are isolated.  Note that system \eqref{sys0} in $\mathbb{R}^3$ is symmetric with respect to the three coordinate planes $x_i=0$, $i=1,2,3$, respectively. Hence, we just need to consider system \eqref{sys0} in $\mathbb{R}_+^3=\{(x_1,x_2,x_3)\in \mathbb{R}^3: \ x_1\geq0, x_2\geq0, x_3\geq0\}$. For convenience, let us define
$$ \sigma_1=\alpha_1(\alpha_3+d_3)+\alpha_2(\alpha_1+d_2)+\alpha_3(\alpha_2+d_1), \quad \sigma_2=\sum_{i=1}^3 \alpha_i+d_i.
$$
We first study the existence and topological classification of equilibria of system \eqref{sys0} in $\mathbb{R}^3_+$. It is easy to see that $O=(0,0,0)$, $e_1=(1,0,0)$, $e_2=(0,1,0)$, $e_3=(0,0,1)$ are equilibria in $\mathbb{R}^3_+$ for any $\alpha_i$, $i=1,2,3$ and $(d_1, d_2, d_3)\in \mathbb{R}^3$. By straightforward computations, we have 

\begin{lemma}[Existence of isolated equilibria]\label{existence of equilibria}System \eqref{sys0} has only isolated equilibria in $\mathbb{R}^3_+$ if and only if $(\alpha_1+d_2)(\alpha_2+d_1)(\alpha_3+d_3)\neq0$. More precisely,
\begin{itemize}
    \item [(i)] if $\alpha_1+d_2, \alpha_2+d_1, \alpha_3+d_3$ have the same sign, then system \eqref{sys0} has  five isolated equilibria $O, e_1, e_2, e_3, Q^*$ in $\mathbb{R}^3_+$, where $Q^*=(q_1^*, q_2^*, q_3^*)$ is a positive equilibrium, here
$$Q^*=\left(\sqrt{\frac{\alpha_3+d_3}{\sigma_2}}, \sqrt{\frac{\alpha_1+d_2}{\sigma_2}}, \sqrt{\frac{\alpha_2+d_1}{\sigma_2}}\right).$$

    \item [(ii)] if at least one of $(\alpha_1+d_2)(\alpha_2+d_1)<0$ and $(\alpha_2+d_1)(\alpha_3+d_3)<0$ holds, system \eqref{sys0} has only four isolated equilibria $O, e_1, e_2, e_3$ in $\mathbb{R}^3_+$.
\end{itemize}
\end{lemma}

To study the topological classification of these isolated equilibria, we compute the associated three eigenvalues as follows.
\begin{table}[H]
		\centering
		\caption{Possible isolated equilibria and the corresponding three eigenvalues }\label{SecType_eq1_atleast7}
		\setlength{\tabcolsep}{10pt}
		\begin{tabular}{ll}
			\toprule
			Equilibrium & three eigenvalues\\
			\midrule
			$O=(0,0,0)$ & $\alpha_1,\,\alpha_2,\,\alpha_3$\\
			$\textbf{e}_1=(1,0,0)$ & $-2\alpha_1,\,\alpha_2+d_1,\,-(\alpha_1+d_2)$\\
			$\textbf{e}_2=(0,1,0)$ & $\,-(\alpha_2+d_1),-2\alpha_2,\,\alpha_3+d_3$\\
			$\textbf{e}_3=(0,0,1)$ & $\,\alpha_1+d_2,\,-(\alpha_3+d_3),-2\alpha_3$\\
            $Q^*=(q_1^*, q_2^*, q_3^*)$ & $\lambda_{Q^*}i,  -\lambda_{Q^*}i, -\frac{2\sigma_1}{\sigma_2}$, here $\lambda_{Q^*}=2\sqrt{\frac{(\alpha_2+d_1)(\alpha_1+d_2)(\alpha_3+d_3)}{\sigma_2}}$\\

			\bottomrule
		\end{tabular}
	\end{table}

Now, we are ready to study the global dynamics when system \eqref{sys0} has only isolated equilibria in $\mathbb{R}^3_+$.
\begin{theorem}[Global dynamics]\label{global dynamics}
If $\alpha_i>0$, $i=1,2,3$ and $(\alpha_1+d_2)(\alpha_2+d_1)(\alpha_3+d_3)\neq0$, then system \eqref{sys0} has exactly two different topological classifications of  global dynamics in $\mathbb{R}^3_+$. More precisely,
\begin{itemize}
\item [(i)] if $\alpha_1+d_2, \alpha_2+d_1, \alpha_3+d_3$ have the same sign, then system \eqref{sys0} has five equiliria: $\{O, e_1, e_2, e_3, Q^*\}$. Moreover, $\mathbb{S}^2_+$ consists of periodic orbits, positive equilibria $Q^*$ and the heteroclinic polycycle $\partial\mathbb{S}^2_+$. And for any $x\in \mathbb{R}^3_+\setminus\{O\}$, $\omega(x)\subset \mathbb{S}^2_+$. The phase portrait is shown in Figure \ref{global dynamcis} (a).

 \item [(ii)] if at least one of $(\alpha_1+d_2)(\alpha_2+d_1)<0$ and $(\alpha_2+d_1)(\alpha_3+d_3)<0$ holds, then system \eqref{sys0} has four equiliria: $\{O, e_1, e_2, e_3\}$ and there exists unique an equilibrium $e_i\in \{e_1, e_2, e_3\}$ such that  for any $x\in \textrm{Int}\mathbb{R}^3_+$, $\omega(x)=\{e_i\}$. The phase portrait is shown in Figure \ref{global dynamcis} (b).
\end{itemize}
\end{theorem}

 \begin{figure}[H]
 \begin{subfigure}{0.5\textwidth}
\includegraphics[width=1.05\linewidth]{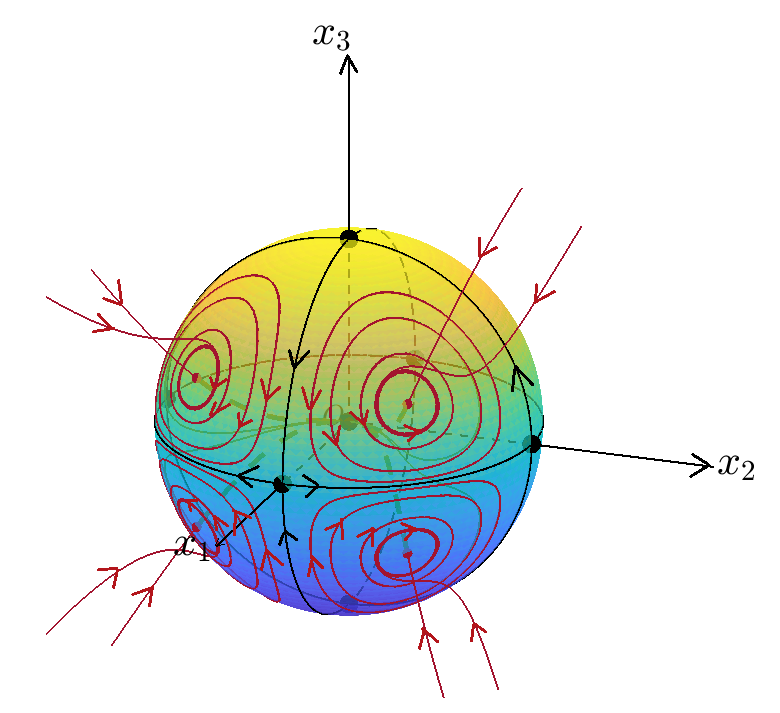}
\caption{ {\footnotesize$\alpha_1+d_2>0, \alpha_2+d_1>0, \alpha_3+d_3>0$}}
  \end{subfigure}
 \begin{subfigure}{0.5\textwidth}
\includegraphics[width=1\linewidth]{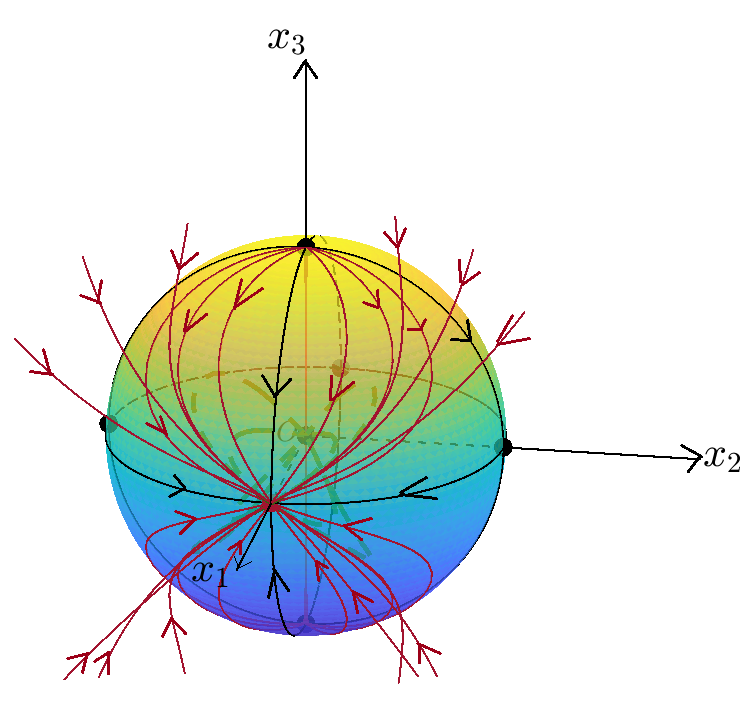}
\caption{{\footnotesize$\alpha_1+d_2>0, \alpha_2+d_1<0, \alpha_3+d_3<0$}}
  \end{subfigure}
\caption{The global dynamics of system \eqref{sys0} with isolated equilibria}
\label{global dynamcis}
\end{figure}
\begin{proof}
(i) Using Lemma \ref{existence of equilibria} it remains to prove that $Q^*$ is a center on $\mathbb{S}^2_+$.  For this, let us consider system \eqref{sys0} restricted on $\mathbb{S}^2_+$, that is,
\begin{equation}\label{reduced}
\begin{cases}
    \dot{x_1}=x_1(-(\alpha_1+d_2)x_1^2-(\alpha_1+\alpha_2+d_1+d_2)x_2^2+(\alpha_1+d_2)),\\
    \dot{x_2}=x_2((\alpha_2+\alpha_3+d_1+d_3)x_1^2+(\alpha_3+d_3)x_2^2-(\alpha_3+d_3)).
\end{cases}
\end{equation}
One can check that
$$H(x_1, x_2)=x_1^{2(\alpha_3+d_3)}x_2^{2(\alpha_1+d_2)}(x_1^2+x_2^2-1)^{\alpha_2+d_1}$$
is a first integral of system \eqref{reduced}. And so, $Q^*$ is a center on $\mathbb{S}^2_+$ by Poincar\'e center Theorem. Taking into account Theorem \ref{attractor} we derive this statement.

(ii) From Lemma \ref{existence of equilibria}, system \eqref{sys0} has four isolated equilibria $O, e_1, e_2, e_3$ in $\mathbb{R}^3_+$ and only one of $e_1, e_2, e_3$
is local asymptotic stable by computation of eigenvalues. Note that $e_1, e_2, e_3$ are on the compact invariant attractive manifold $\mathbb{S}^2_+$ by Theorem \ref{attractor}. So any $x\in \textrm{Int}\mathbb{R}^3_+$, $\omega(x)=\{e_i\}$.
\end{proof}

\section{System \eqref{sys0} driven by linear multiplicative noise}\label{stochastic 3d}

In this Section, we consider the stochastic dynamics of system \eqref{sys0} under the perturbation of linear multiplicative Wiener noise, that is the following system:
\begin{equation}\label{sys5}
\begin{cases}
dx_1(t)=x_1(\alpha_1-\alpha_1 x_1^2-(\alpha_1+\alpha_2+d_1)x_2^2+d_2x_3^2)dt+\sigma x_1dW_t,\\
dx_2(t)=x_2(\alpha_2+d_1x_1^2-\alpha_2 x_2^2-(\alpha_2+\alpha_3+d_3)x_3^2)dt+\sigma x_2dW_t,\\
dx_3(t)=x_3(\alpha_3-(\alpha_3+\alpha_1+d_2)x_1^2+d_3x_2^2-\alpha_3 x_3^2)dt+\sigma x_3dW_t,
\end{cases}
\end{equation}
where $(x_1, x_2, x_3)\in \mathbb{R}^3$, $\sigma>0$ represents the strength of noise,
$(W_t)$ is the Wiener process, $\alpha_i>0$, $d_i\in \mathbb{R}$, $i=1,2,3$.

For convenience, we first give some useful notations. Let $b(x)$ be the drift term of system \eqref{sys5} and $(a^{ij})$ the diffusion matrix, i.e., $a_{ii}=\sigma^2x_i^2$, $i=1,2,3$ and $a_{ij}=0$ if $i\neq j$.
We rewrite the drift term of system \eqref{sys5} into the following form:
$$dx_i=x_i(\alpha_i+\sum_{j=1}^3b_{ij}x^2_{j})dt.$$

We first show the existence of global solutions of stochastic system \eqref{sys5}.
\begin{theorem}[Existence of global solutions]\label{thm10}
For any $x\in \mathbb{R}^3$ and almost surely $\omega\in \Omega$, there exists a global unique solution $\Phi(\cdot, \omega, x)$ to \eqref{sys5} with initial data $x$.
\end{theorem}

\begin{proof}
Define the Lyapunov function $V: \mathbb{R}^3\rightarrow R_+$ by
$$ V(x):=x_1^2+x_2^2+x_3^2,$$
and the operator $\mathscr{L}$ by
\begin{equation}\label{Fop}
\mathscr{L}f(x):= \langle \nabla f(x), b(x) \rangle + \frac12 a^{ij}\partial^2_{ij}f(x),\ \  f\in C^2(\mathbb{R}^3),
\end{equation}
where $a^{ij}$ is the diffusion matrix of \eqref{sys5}.
Then, by some computations,
\begin{align*}
\mathscr{L} V(x) & = 2\langle x, b(x)\rangle+ \sigma^2 \sum_{i=1}^3x_i^2\\
& = -2(\alpha_1 x_1^2+\alpha_2 x_2^2+\alpha_3 x_3^2)(x_1^2+x_2^2+x_3^2-1)+\sigma^2 \sum_{i=1}^3x_i^2 \\
&\leq -2\min\{\alpha_1, \alpha_2, \alpha_3\}\|x\|^4+(2\max\{\alpha_1, \alpha_2, \alpha_3\}+\sigma^2)\|x\|^2 \\
&= V(x)(-2\min\{\alpha_1, \alpha_2, \alpha_3\}\|x\|^2+2\max\{\alpha_1, \alpha_2, \alpha_3\}+\sigma^2).
\end{align*}
Therefore, we get
$$\mathscr{L} V(x)\leq (2\max\{\alpha_1, \alpha_2, \alpha_3\}+\sigma^2)V(x).$$
Using Theorem 3.3.5 in \cite{krz} we derive the global existence and uniqueness of the solution to \eqref{sys5}.
\end{proof}

Now we state our main result as follows.
\begin{theorem}[Stochastic dynamics]\label{stochastic dynamics}
Let $\alpha_i>0, i=1,2,3$ and assume that $d_1\le 0, d_3\le0$ and $\alpha_1+\alpha_3+d_2\ge 0$. Then, there exists a threshold $\sigma_0=\sqrt{2\max\{\alpha_1, \alpha_2, \alpha_3\}}$ such that when $\sigma>\sigma_0$, the noise destroys the attracting invariant sphere $\mathbb{S}^2$. And the change of noise intensity leads to transitions of stationary measures. More precisely,
\begin{itemize}
 \item [(i)] if $\sqrt{2\max\{\alpha_1, \alpha_2, \alpha_3\}}<\sigma$, then for any $x\in \mathbb{R}^3$, $\Phi(t, \omega, x)\to O$ as $t\to \infty$ for almost surely $\omega\in \Omega$. And $\delta_O$ is the unique stationary measure of system \eqref{sys5}.
 \item [(ii)] if $\sqrt{2\min\{\alpha_1, \alpha_2, \alpha_3\}}< \sigma<\sqrt{2\max\{\alpha_1, \alpha_2, \alpha_3\}}$, then system \eqref{sys5} has at least two stationary measures: one is $\delta_O$ and the other is supported on a ray.
 \item[(iii)] if $0<\sigma<\sqrt{2\min\{\alpha_1, \alpha_2, \alpha_3\}}$, then \eqref{sys5} has at least four stationary measures: one is $\delta_O$ and the others are supported on rays.
\end{itemize}
\end{theorem}
{\bf The proof of Theorem \ref{stochastic dynamics}:}
We first claim that: if there exist  $i\in \{1,2,3\}$ such that $\alpha_i<\frac12 \sigma^2$, then for any $x\in \mathbb{R}^3$, $x_i(t, \omega, x)\to 0$ as $t\to \infty$ for almost surely $\omega\in \Omega$. For this purpose,
let $y=(y_1, y_2, y_3)$ and
$y_i(t,\omega,x_0):=e^{-(\alpha_i-\frac12 \sigma^2)t-\sigma W_t}x_i(t,\omega,x_0)$,
$t\geq 0$, $\omega\in \Omega$, $x_0\in \mathbb{R}^3$.
Then \eqref{sys5} becomes
\begin{equation}\label{change2}
dy_i=y_i(\sum_j m_jb_{ij} y_j^2)dt,
\ \ i=1,2,3,
\end{equation}
where $m_j=\exp\{2(\alpha_j-\frac12\sigma^2)t+2\sigma W_t(\omega)\}$.
Define the Lyapunov function $V: \mathbb{R}^3\to \mathbb{R}_+$ by
\begin{equation}
V(y):=\sum_{i=1}^3y_i^2,\ \  (y_1, y_2, y_3)\in \mathbb{R}^3.
\end{equation}
Then, along the trajectory of \eqref{change2} with initial data $x_0\neq O$ we compute
\begin{align}\label{Lyapunov_1}
    \frac{dV(y(t))}{dt}&=2\sum_iy_i^2(\sum_jb_{ij}m_jy_j^2)\\
    &= 2I_0+2I_1+2I_2+2I_3,
\end{align}
where
\begin{equation}
\begin{cases}
I_0=-\sum_{i=1}^3\alpha_im_i y_i^4, \\
I_1=(d_1m_1-(\alpha_1+\alpha_2+d_1)m_2)y_1^2y_2^2,\\
I_2=(d_2m_3-(\alpha_1+\alpha_3+d_2)m_1)y_1^2y_3^2, \\
I_3=(d_3m_2-(\alpha_2+\alpha_3+d_3)m_3)y_2^2y_3^2.
\end{cases}
\end{equation}
Without loss of generality, we assume that $\alpha_1\ge\alpha_2\ge\alpha_3>0$, and so, $m_1\ge m_2\ge m_3>0$ for any $t\ge0$ and $\omega\in \Omega$. Therefore, we have
\begin{equation}\label{I_0}
I_0\le-m_3(\alpha_1y_1^4+\alpha_2y_2^4+\alpha_3y_3^4).
\end{equation}
Now we estimate $I_i$, $i=1,2,3$.

Since $d_1\le0$ and $\alpha_1+\alpha_2>0$, we have $d_1m_1\le d_1m_2$ and $-(\alpha_1+\alpha_2)m_2\le -(\alpha_1+\alpha_2)m_3$. Thus, \begin{align}\label{I_1}
I_1&\le (d_1m_2-(\alpha_1+\alpha_2+d_1)m_2)y_1^2y_2^2, \\ \notag
&=-(\alpha_1+\alpha_2)m_2y_1^2y_2^2, \\ \notag
&\le -(\alpha_1+\alpha_2)m_3y_1^2y_2^2.
\end{align}

Note that $\alpha_1+\alpha_3+d_2\ge 0$, we derive
\begin{align}\label{I_2}
I_2&\le (d_2m_3-(\alpha_1+\alpha_3+d_2)m_3)y_1^2y_3^2, \\ \notag
&=-(\alpha_1+\alpha_3)m_3y_1^2y_3^2.
\end{align}

Since $d_3\le0$, we obtain
\begin{align}\label{I_3}
I_3&\le (d_3m_3-(\alpha_2+\alpha_3+d_3)m_3)y_2^2y_3^2, \\ \notag
&=-(\alpha_2+\alpha_3)m_3y_2^2y_3^2.
\end{align}
Thus, combined with estimations \eqref{I_0}-\eqref{I_3}, one derive
\begin{align*}
    \frac{dV(y(t))}{dt}&\le -2m_3(\sum_{i=1}^3\alpha_iy_i^4+(\alpha_1+\alpha_2)y_1^2y_2^2+(\alpha_1+\alpha_3)y_1^2y_3^2+(\alpha_2+\alpha_3)y_2^2y_3^2)\\
    &= -2m_3(y_1^2+y_2^2+y_3^2)(\alpha_1y_1^2+\alpha_2y_2^2+\alpha_3y_3^2)\\
    &\le -2m_3\alpha_3V^2<0.
\end{align*}
This yields that
\begin{equation}\label{esss}
\sup\limits_{t\geq 0}\|y(t, \omega, x_0)\|< \infty, \ \  \mathbb{P}-a.s.
\end{equation}
Note that
\begin{equation}\label{xxx}
x_i(t,\omega,x_0)=e^{(\alpha_i-\frac12\sigma^2)t+\sigma W_t}y_i(t, \omega,x_0).
\end{equation}
Therefore, if there exists $i\in\{1,2,3\}$ such that $\alpha_i<\frac12\sigma^2$, then,  by \eqref{xxx}, we obtain $\mathbb{P}$-a.s.
 $$x_i(t, \omega, x_0)\rightarrow 0 \ \  as\,\, t\rightarrow \infty,$$
which completes the claim.

(i) Since $\max\{\alpha_1, \alpha_2,\alpha_3\}<\frac12\sigma^2$, taking into account \eqref{xxx} one immediately has $\mathbb{P}$-a.s.
 $$x(t, \omega, x_0)\rightarrow O \ \  as\,\, t\rightarrow \infty,$$
which yields that the noise destroys the attracting invariant sphere $\mathbb{S}^2.$

It remains to prove that $\delta_O$ is the unique stationary measure of system \eqref{sys5} in $\mathbb{R}^3$. For this purpose, first note that $O$ is a random equilibrium, and so $\delta_O$ is an ergodic stationary measure. Taking into account that for any $x\in \mathbb{R}^3$, $\Phi(t, \omega,x)\to O$ almost surely and using Lebesgue-dominated convergence theorem, for any $f\in C_b(\mathbb{R}^3)$, we derive
\begin{align*}
\displaystyle\lim\limits_{t\rightarrow \infty}\int_{\mathbb{R}^3}f(z)P(t,x, dz)
= \displaystyle\lim\limits_{t\rightarrow \infty}\int_{\Omega}f(\Phi(t, \omega, x))\mathbb{P}(d\omega)
=\int_{\mathbb{R}^3}f(z)\delta_O (dz),
\end{align*}
which implies that
\begin{equation}\label{es 1}
\displaystyle\lim\limits_{t\rightarrow \infty} P(t, x, \cdot)\rightarrow \delta_O  \ \  \textrm{weakly} \,\, \textrm{in} \,\, \mathcal{P}(\mathbb{R}^3).
\end{equation}

Now we are ready to prove the uniqueness of $\delta_O$. For this, we use the same analysis as in the proof of Theorem 1.1 in \cite{zhou2024stochastic}. Assume that $\nu\in \mathcal{P}(\mathbb{R}^3)$ is another ergodic stationary measure such that $\nu(\cdot)\ne \delta_O(\cdot)$.
Then, taking into account \eqref{es 1} we derive
\begin{equation}\label{es_3}
\int_{\mathbb{R}^3}P(t, x, \cdot)\nu(dx)\stackrel{w}{\rightharpoonup}\delta_O(\cdot), \ \ as\ t \rightarrow \infty.
\end{equation}
However, using the definition of stationary measures, for any $t\ge 0$, one has $$\int_{\mathbb{R}^3}P(t, x, \cdot)\nu(dx)=\nu(\cdot),$$ which violates \eqref{es_3}.

(ii) Without loss of generality, we assume that $\alpha_1\ge\alpha_2\ge\alpha_3>0$, and so $\sqrt{2\alpha_3}<\sigma<\sqrt{2\alpha_1}$. Note that $H_1:=\{(x_1, x_2, x_3): x_2=x_3=0, x_1>0\}$ is invariant under system \eqref{sys5}. So, we consider the restriction of  system \eqref{sys5} on $H_1$, that is,
\begin{equation}\label{reduced-system}
\frac{dx_1}{dt}=x_1(\alpha_1-\alpha_1x_1^2)dt+\sigma x_1 dW_t.
\end{equation}
Since $\sqrt{2\alpha_3}<\sigma<\sqrt{2\alpha_1}$, applying Lemma 3.4 and 3.5 in \cite{zhou2024stochastic} system \eqref{reduced-system} has two stationary measures: $\delta_O$ and $\mu_1$ which is supported on $H_1$. This yields that system \eqref{sys5} has at least two stationary measures.

(iii) Note that
$H_2=\{(x_1, x_2, x_3): x_2>0, x_1=x_3=0\}$ and $H_3=\{(x_1, x_2,x_3): x_1=x_2=0, x_3>0\}$ and $H_1$ are invariant under system \eqref{sys5}. Since $0<\sigma<\sqrt{2\min\{\alpha_1, \alpha_2, \alpha_3\}}$, restricting system \eqref{sys5} on $H_i$ and  applying Lemma 3.4 and 3.5 in \cite{zhou2024stochastic} again there exists a nontrivial stationary measure denoted by $\mu_i$ supported on the positive $x_i$-axis, for each $i=1,2,3.$ Thus, system \eqref{sys5} has at least 4 stationary measures: $\delta_O$, $\mu_i$, $i=1,2,3.$
\hfill $\square$

\section*{Acknowledgments}
		The first and third authors are supported by the  National Natural Science Foundation of China (No. 12322109).


\begin{thebibliography}{10}

\bibitem{StrangeAttractor}
Alain Arneodo, Pierre Coullet, and Charles Tresser.
\newblock Occurence of strange attractors in three-dimensional volterra equations.
\newblock {\em Physics Letters A}, 79(4):259--263, 1980.

\bibitem{busse1}
F.~H. Busse and K.~E. Heikes.
\newblock Convection in a rotating layer: A simple case of turbulence.
\newblock {\em Science.}, 208(4440):173--175, 1980.

\bibitem{busse2}
Friedrich~H. Busse.
\newblock An example of a direct bifurcation into a turbulent state.
\newblock In {\em Nonlinear dynamics and turbulence}, Interaction of Mechanics and Mathematics Series, pages 93--100. Pitman, Boston, 1983.

\bibitem{Li-Michael}
Xiaochun Cao, Zhen Jin, Guirong Liu, and Michael~Y. Li.
\newblock On the basic reproduction number in semi-{M}arkov switching networks.
\newblock {\em J. Biol. Dyn.}, 15(1):73--85, 2021.

\bibitem{jiang1}
Lifeng Chen, Zhao Dong, Jifa Jiang, Lei Niu, and Jianliang Zhai.
\newblock Decomposition formula and stationary measures for stochastic {L}otka-{V}olterra system with applications to turbulent convection.
\newblock {\em J. Math. Pures Appl. (9)}, 125:43--93, 2019.

\bibitem{CM21}
Michele Coti~Zelati and Martin Hairer.
\newblock A noise-induced transition in the {L}orenz system.
\newblock {\em Comm. Math. Phys.}, 383(3):2243--2274, 2021.

\bibitem{Crauel-Flandoli}
Hans Crauel and Franco Flandoli.
\newblock Additive noise destroys a pitchfork bifurcation.
\newblock {\em J. Dynam. Differential Equations}, 10(2):259--274, 1998.

\bibitem{Darboux}
Gaston Darboux.
\newblock M{\'e}moire sur les {\'e}quations diff{\'e}rentielles alg{\'e}briques du premier ordre et du premier degr{\'e}.
\newblock {\em Bulletin des sciences math{\'e}matiques et astronomiques}, 2(1):151--200, 1878.

\bibitem{busse3}
KE~Heikes and Friedrich~H Busse.
\newblock Weakly nonlinear turbulence in a rotating convection layer.
\newblock {\em Annals of the New York Academy of Sciences}, 357:28--36, 1980.

\bibitem{Economi}
Arlindo Kamimura, Geraldo~F Burani, and Humberto~M Fran{\c{c}}a.
\newblock The economic system seen as a living system: a lotka-volterra framework.
\newblock {\em Emergence: Complexity and Organization}, 13(3):80, 2011.

\bibitem{krz}
Rafail Khasminskii.
\newblock {\em Stochastic stability of differential equations}, volume~66 of {\em Stochastic Modelling and Applied Probability}.
\newblock Springer, Heidelberg, second edition, 2012.
\newblock With contributions by G. N. Milstein and M. B. Nevelson.

\bibitem{Knebel1}
Johannes Knebel, Torben Kr{\"u}ger, Markus~F Weber, and Erwin Frey.
\newblock Coexistence and survival in conservative lotka-volterra networks.
\newblock {\em Physical review letters}, 110(16):168106, 2013.

\bibitem{kl}
Aleksandr Kolmogorov.
\newblock Sulla teoria di volterra della lotta per lesistenza.
\newblock {\em Gi. Inst. Ital. Attuari}, 7:74--80, 1936.

\bibitem{Smale_1976}
S.~Smale.
\newblock On the differential equations of species in competition.
\newblock {\em J. Math. Biol.}, 3(1):5--7, 1976.

\bibitem{zhou2024stochastic}
Dongmei Xiao, Deng Zhang, and Chenwan. Zhou.
\newblock Stochastic bifurcation of a three-dimensional stochastic kolmogorov system.
\newblock {\em ArXiv:2408.01560}, 2024.

\end{thebibliography}


\begin{thebibliography}{10}

\bibitem{StrangeAttractor}
A. Arneodo, P. Coullet, and C. Tresser.
\newblock Occurence of strange attractors in three-dimensional volterra equations.
\newblock {\em Physics Letters A}, 79(4):259--263, 1980.

\bibitem{Arnold}
L. Arnold.
\newblock  {\em Random Dynamical Systems.}
\newblock Springer, Berlin, 1998.

\bibitem{braza}
Z.~Brze{\'z}niak, T.~Komorowski, and S.~Peszat.
\newblock Ergodicity for stochastic equations of Navier--Stokes type.
\newblock {\em Electronic Communications in Probability}, 27:1-10, 2022.



\bibitem{busse1}
F.~H. Busse and K.~E. Heikes.
\newblock Convection in a rotating layer: A simple case of turbulence.
\newblock {\em Science.}, 208(4440):173--175, 1980.

\bibitem{busse2}
F.~H. Busse.
\newblock An example of a direct bifurcation into a turbulent state.
\newblock {\em Nonlinear dynamics and turbulence}, Interaction of Mechanics and Mathematics Series, pages 93--100. Pitman, Boston, 1983.

\bibitem{Doss}
H. Doss.
\newblock Liens entre équations différentielles stochastiques et ordinaires.
\newblock {\em Ann. Inst. H. Poincaré Sect. B (N.S.)}, 13 (2),  99–125, 1977.



\bibitem{Li-Michael1}
Y. M. Li.
    \newblock Dulac criteria for autonomous systems having an invariant affine manifold.
   \newblock {\em J. Math. Anal. Appl.} 199, 374 - 390,  1996.

 \bibitem{Li-Michael2}
Y. M. Li, J. S. Muldowney.
\newblock  Dynamics of differential equations on invariant manifolds.
\newblock {\em J. Differential Equations}, 168(2): 295 - 320, 2000.


\bibitem{jiang1}
L. Chen, Z. Dong, J. Jiang, L. Niu, and J. Zhai.
\newblock Decomposition formula and stationary measures for stochastic {L}otka-{V}olterra system with applications to turbulent convection.
\newblock {\em J. Math. Pures Appl. (9)}, 125:43--93, 2019.


\bibitem{Crauel-Flandoli}
H. Crauel and F. Flandoli.
\newblock Additive noise destroys a pitchfork bifurcation.
\newblock {\em J. Dynam. Differential Equations}, 10(2):259--274, 1998.

\bibitem{Darboux}
G. Darboux.
\newblock M{\'e}moire sur les {\'e}quations diff{\'e}rentielles alg{\'e}briques du premier ordre et du premier degr{\'e}.
\newblock {\em Bulletin des sciences math{\'e}matiques et astronomiques}, 2(1):151--200, 1878.

\bibitem{busse3}
K. Heikes and F.~H. Busse.
\newblock Weakly nonlinear turbulence in a rotating convection layer.
\newblock {\em Annals of the New York Academy of Sciences}, 357:28--36, 1980.

\bibitem{HJLY1}
W. Huang, M. Ji, Z. Liu, and Y. Yi.
\newblock Integral identity and measure estimates for stationary Fokker-Planck equations.
\newblock {\em Ann. Probab.} 43 (4),  1712 - 1730, 2015.

\bibitem{HJLY2}
W. Huang, M. Ji, Z. Liu, and Y. Yi.
\newblock Steady states of Fokker-Planck equations: I. Existence
\newblock {\em J. Dynam. Differential Equations} 27 (3 - 4),  721 - 742, 2015.


\bibitem{HJLY4}
W. Huang, M. Ji, Z. Liu, and Y. Yi.
\newblock Concentration and limit behaviors of stationary measures.
\newblock {\em Phys. D}, 369, 1 - 17,  2018.

\bibitem{krz}
R. Khasminskii.
\newblock  Stochastic stability of differential equations. 
\newblock With contributions by G. N. Milstein and M. B. Nevelson.
\newblock {\em Springer, Heidelberg}, second edition, 2012.


\bibitem{kl}
A. Kolmogorov.
\newblock Sulla teoria di volterra della lotta per lesistenza.
\newblock {\em Gi. Inst. Ital. Attuari}, 7:74--80, 1936.


\bibitem{Sussmann}
H. Sussmann.
\newblock On the gap between deterministic and stochastic ordinary differential equations.
\newblock {\em  Ann. Probability}, 6(1):19--41, 1978.

\bibitem{zhou2024stochastic}
D. Xiao, D. Zhang, and C. Zhou.
\newblock Stochastic bifurcation of a three-dimensional stochastic kolmogorov system.
\newblock {\em ArXiv:2408.01560}, 2024.

\bibitem{Zhao}
J. Zhao, J. Shen, and K. Lu.
\newblock  Persistence of  $C^1$  inertial manifolds under small random perturbations
\newblock {\em J. Dynam. Differential Equations} 36, S333–S385, 2024.

\end{thebibliography}
\end{document}